\documentclass{amsart}
\usepackage{amsmath}
\usepackage{amssymb}
\usepackage{amscd}
\usepackage{graphicx}
\usepackage{setspace}
\usepackage[table, xcdraw]{xcolor}

\newcommand{\nc}{\newcommand}
\nc{\on}{\operatorname}
\nc{\BR}{\mathbb R}
\nc{\BF}{\mathbb F}
\nc{\BC}{\mathbb C}
\nc{\BQ}{\mathbb Q}
\nc{\BZ}{\mathbb Z}
\nc{\BN}{\mathbb N}
\nc{\ii}{\item}
\nc{\abs}[1]{\left\lvert #1 \right\rvert}
\nc{\tensor}{\otimes}
\nc{\Hom}{\on{Hom}}
\nc{\End}{\on{End}}
\nc{\wt}{\widetilde}
\nc{\can}{\on{can}}
\nc{\id}{\on{Id}}

\nc{\union}{\cup}
\nc{\intersect}{\cap}
\nc{\dunion}{\displaystyle\bigcup}
\nc{\dintersect}{\displaystyle\bigcap}
\nc{\dint}{\displaystyle\int}
\nc{\dsum}{\displaystyle\sum}
\nc{\dprod}{\displaystyle\prod}
\nc{\mf}{\mathfrak}
\nc{\mc}{\mathcal}

\nc{\idp}{\mathfrak{p}}
\nc{\idq}{\mathfrak{q}}
\nc{\ida}{\mathfrak{a}}
\nc{\idb}{\mathfrak{b}}

\nc{\inc}{\subseteq}
\nc{\ninc}{\nsubseteq}

\nc{\seq}[1]{{#1}_1, {#1}_2, \ldots, {#1}_n}
\nc{\seqq}[2]{{#1}_1{#2} {#1}_2{#2} \ldots {#2} {#1}_n}
\nc{\fl}[1]{\left\lfloor #1 \right\rfloor}
\nc{\ce}[1]{\left\lceil #1 \right\rceil}
\nc{\Borel}{\mathfrak{B}}

\nc{\toname}[1]{\overset{#1}{\to}}

\newtheorem{theorem}{Theorem}[section]
\newtheorem{lemma}[theorem]{Lemma}
\newtheorem{proposition}[theorem]{Proposition}
\newtheorem{corollary}[theorem]{Corollary}
\newtheorem*{conjecture}{Conjecture}
\newtheorem*{question}{Question}

\newenvironment{definition}[1][Definition]{\begin{trivlist}
\item[\hskip \labelsep {\bfseries #1}]}{\end{trivlist}}

\nc{\jacobi}[2]{\left(\frac{#1}{#2}\right)}
\begin{document}
\title{Even $(\bar{s}, \bar{t})$-core partitions and self-associate characters of $\tilde{S}_n$.}
\author{Calvin Deng}
\begin{abstract}
    A partition is a $\bar{s}$-core if it is the result of removing all of the $s$-bars from a partition. We extend a method of Olsson and Bessenrodt to determine the number of even partitions that are simultaneously $\bar{s}$-core and $\bar{t}$-core. When $p$ and $q$ are distinct primes, this also determines the number of self-associate characters of $\tilde{S}_n$ that are simultaneosly defect 0 for $p$ and $q$.
\end{abstract}
\maketitle
\section{Introduction}
Navarro and Willems~\cite{navarro} were the first to investigate the question
of when a block was simultaneously a $p$-block and a $q$-block for a group $G$
and odd primes $p$ and $q$. In particular, they conjectured that if a $p$-block
and $q$-block coincided, then that block consists of a single character. While
Bessenrodt disproved this conjecture, Olsson and Stanton~\cite{stanton}
showed that it was true in the case where $G$ was the symmetric group $S_n$.
Olsson and Bessenrodt~\cite{Bessenrodt20063} then showed that Navarro's conjecture was true in the case where $G$ was the spin symmetric group $\tilde{S}_n$.

There has been considerable work enumerating characters with defect 0 in $p$
and $q$, as well as enumerating the subclass of self-associate characters.
Anderson~\cite{anderson} showed that the number of characters of $S_n$ with
defect 0 in $p$ and $q$ is $\frac{1}{p + q}\binom{p + q}{p}$, while Ford, Mai,
and Sze~\cite{ford} showed that the number of self-associate characters of
$S_n$ with defect 0 in $p$ and $q$ is $\binom{\frac{p - 1}{2} + \frac{q -
1}{2}}{\frac{p - 1}{2}}$. On the other hand, Bessenrodt and
Olsson~\cite{Bessenrodt20063} showed that the number of spin characters in $\tilde{S}_n$ with defect 0 for $p$ and $q$ was also $\binom{\frac{p - 1}{2} + \frac{q - 1}{2}}{\frac{p - 1}{2}}$. 
However, they left the problem of enumerating the number of self-associate spin characters of
$\tilde{S}_n$ with defect 0 for $p$ and $q$ unresolved.

Much of this work has concentrated on the partitions that correspond combinatorially to $p$-blocks for $S_n$ and $\tilde{S}_n$.
In the case of $S_n$, they are the $p$-core partitions; in the case of
$\tilde{S}_n$, they are the $\bar{p}$-core partitions. Furthermore, a character
of $S_n$ is self-associate if and only if the corresponding partition is
self-conjugate, while a spin character of $\tilde{S}_n$ is self-associate if and only if the corresponding partition is even.
In this paper, we will enumerate the self-associate spin characters in $\tilde{S}_n$ with defect 0 for $p$ and $q$ by counting the number of even $(\bar{s}, \bar{t})$-cores for relatively prime odd integers $s$ and $t$ greater than 1.
\begin{theorem}
    Suppose $s, t > 1$ are relatively prime odd positive integers.
    \begin{itemize}
        \item If $s, t\equiv 3\pmod{4}$, the number of even $(\bar{s}, \bar{t})$-core partitions is
            \[
                \frac{1}{2}\binom{\frac{s - 1}{2} + \frac{t - 1}{2}}{\frac{s - 1}{2}}.
            \]
        \item Otherwise, the number of even $(\bar{s}, \bar{t})$-core partitions is
            \[
                \frac{1}{2}\left( \binom{\frac{s - 1}{2} + \frac{t -
                1}{2}}{\frac{s - 1}{2}} + {(-1)}^{(s - 1)(t - 1) / 8}
                \left(\frac{s}{t}\right)\binom{\fl{s / 4} + \fl{t / 4}}{\fl{t /
                4}} \right).
            \]
            where $\jacobi{s}{t}$ denotes the Jacobi symbol.
    \end{itemize}
\end{theorem}
Our proof will use the bijection between spin characters of defect 0 for $p$
and $q$ in $\tilde{S}_n$ and monotone $\left(\frac{p - 1}{2},
\frac{q-1}{2}\right)$ paths. In particular, self-associate spin characters will
correspond to paths of a certain ``parity''.

One immediate consequence is the following corollary:
\begin{corollary}
    Suppose $p$ and $q$ are distinct odd primes. 
    \begin{itemize}
        \item If $p, q\equiv 3\pmod{4}$, the number of self-associate spin characters of $\tilde{S}_n$ which are simultaniously of defect 0 for $p$ and $q$ is 
            \[
                \frac{1}{2}\binom{\frac{p - 1}{2} + \frac{q - 1}{2}}{\frac{p - 1}{2}}.
            \]
        \item Otherwise, the number of self-associate spin characters of $\tilde{S}_n$ which are simultaniously of defect 0 for $p$ and $q$ is 
            \[
                \frac{1}{2}\left( \binom{\frac{p - 1}{2} + \frac{q - 1}{2}}{\frac{p - 1}{2}} + (-1)^{(p - 1)(q - 1) / 8} \left(\frac{p}{q}\right)\binom{\fl{p / 4} + \fl{q / 4}}{\fl{q / 4}} \right)
            \]
    \end{itemize}
\end{corollary}
In addition, spin characters of $\tilde{S}_n$ split on restriction to $\tilde{A}_n$ if and only if the spin character is self-associate. Thus we also immediately get the number of spin characters of $\tilde{A}_n$ that are of defect 0 for $p$ and $q$.
\begin{corollary}
    Suppose $p$ and $q$ are distinct odd primes.
    \begin{itemize}
        \item If $p, q\equiv 3\pmod{4}$, the number of spin characters of $\tilde{A}_n$ which are simultaniously of defect 0 for $p$ and $q$ is 
            \[
                \frac{3}{2}\binom{\frac{p - 1}{2} + \frac{q - 1}{2}}{\frac{p - 1}{2}}.
            \]
        \item Otherwise, the number of spin characters of $\tilde{A}_n$ which are simultaniously of defect 0 for $p$ and $q$ is 
            \[
                \frac{3}{2}\binom{\frac{p - 1}{2} + \frac{q - 1}{2}}{\frac{p - 1}{2}} + \frac{1}{2}(-1)^{(p - 1)(q - 1) / 8} \left(\frac{p}{q}\right)\binom{\fl{p / 4} + \fl{q / 4}}{\fl{q / 4}}
            \]
    \end{itemize}
\end{corollary}
\begin{proof}
    By the result of Bessenrodt and Olsson, there are $\binom{\frac{p - 1}{2} +
    \frac{q - 1}{2}}{\frac{p - 1}{2}}$ spin characters of $\tilde{S}_n$ that
    are defect 0 for $p$ and $q$. The corollary then follows immediately from
    the fact that the number of spin characters of $\tilde{A}_n$ (with defect 0
    for $p$ and $q$) is equal to the number of spin characters of $\tilde{S}_n$
    (with defect 0 for $p$ and $q$) plus the number of spin characters of
    $\tilde{S}_n$ (with defect 0 for $p$ and $q$) that split upon restriction to $\tilde{A}_n$.
\end{proof}
\section{Preliminaries}
Recall the definition of an $\bar{s}$-core partition and a $(\bar{s},\bar{t})$-core partition.
\begin{definition}
    A bar partition $\lambda = (\lambda_1, \lambda_2, \ldots, \lambda_k)$ is a partition satisfying $\lambda_1 > \cdots > \lambda_k$.
\end{definition}
Note that since the constituents of $\lambda$ are distinct, we can talk about positive integers $a$ either being in $\lambda$ or not being in $\lambda$, much like a set.

The definition of a $\bar{s}$-core partition is a bit involved, and for the technical details, see~\cite{hoffman}. However, we will use the following (equivalent) combinatorial specification.
\begin{proposition}
    A bar partition $\lambda$ is a $\bar{s}$-core partition if and only if the following three constraints hold:
    \begin{itemize}
        \item
            No part in $\lambda$ is divisible by $s$.
        \item
            For all $a$ in $\lambda$ with $a > s$, $a - s\in\lambda$.
        \item
            For all $a$ in $\lambda$ with $1\le a\le s - 1$, $s - a\notin\lambda$.
    \end{itemize}
\end{proposition}
\begin{definition}
    A partition $\lambda$ is a \textit{$(\bar{s},\bar{t})$-core partition} if it is both a $\bar{s}$-core partitions and a $\bar{t}$-core partition.
\end{definition}
\begin{definition}
    A partition $\lambda$ is said to be \textit{even} if an even number of its parts are even.
\end{definition}
Note that a partition $\lambda$ being even is equivalent to $\abs{\lambda} - \ell_\lambda$ being even, where $\abs{\lambda}$ and $\ell_\lambda$ denote the size and length of $\lambda$, respectively.
\begin{proposition}
    Associate classes of spin characters of $\tilde{S}_n$ correspond to bar partitions of $n$. In addition, 
    \begin{itemize}
        \item Associate classes classes with defect 0 for $p$ correpond to $\bar{p}$-core partitions.
        \item Self-associate spin characters of $\tilde{S}_n$ correspond to even partitions. 
    \end{itemize}
\end{proposition}
    For a proof (and more background on the theory of characters of
    $\wt{S}_n$), see Hoffman, \S 10.~\cite{hoffman}
\section{Yin-Yang Diagrams}
From now on we assume that $s, t > 1$ are relatively prime odd integers. Set $m = \fl{s/2}, n = \fl{t/2}, a = \fl{s / 4}, b = \fl{t / 4}$.

The Yin-Yang diagram of $(s, t)$ can be represented as an $m\times n$ grid of
integers. If the lower left corner of the grid is $(0,0)$ and upper right
corner is $(n, m)$, then for each $1\le x\le n$ and $1 \le y \le m$, we place
the value $\abs{sx - ty}$ in the square whose upper right corner is $(x,y)$.
The Yin half of the diagram corresponds to the region of the Yin-Yang diagram
corresponding to ordered pairs $(x,y)$ where $sx - ty < 0$; the Yang half corresponds to the portion of the diagram where $sx - ty > 0$.

\begin{figure}
    \label{yinyang}
    \centering
    \begin{tabular}{|>{\columncolor[HTML]{000000}}l |lllllll}
        \cline{1-7}
        {\color[HTML]{FFFFFF} 89} & \multicolumn{1}{l|}{\cellcolor[HTML]{000000}{\color[HTML]{FFFFFF} 76}} & \multicolumn{1}{l|}{\cellcolor[HTML]{000000}{\color[HTML]{FFFFFF} 63}} & \multicolumn{1}{l|}{\cellcolor[HTML]{000000}{\color[HTML]{FFFFFF} 50}} & \multicolumn{1}{l|}{\cellcolor[HTML]{000000}{\color[HTML]{FFFFFF} 37}} & \multicolumn{1}{l|}{\cellcolor[HTML]{000000}{\color[HTML]{FFFFFF} 24}} & \multicolumn{1}{l|}{\cellcolor[HTML]{000000}{\color[HTML]{FFFFFF} 11}} & 2  \\ \cline{1-7}
        {\color[HTML]{FFFFFF} 72} & \multicolumn{1}{l|}{\cellcolor[HTML]{000000}{\color[HTML]{FFFFFF} 59}} & \multicolumn{1}{l|}{\cellcolor[HTML]{000000}{\color[HTML]{FFFFFF} 46}} & \multicolumn{1}{l|}{\cellcolor[HTML]{000000}{\color[HTML]{FFFFFF} 33}} & \multicolumn{1}{l|}{\cellcolor[HTML]{000000}{\color[HTML]{FFFFFF} 20}} & \multicolumn{1}{l|}{\cellcolor[HTML]{000000}{\color[HTML]{FFFFFF} 7}}  & 6                                                                      & 19 \\ \cline{1-6}
        {\color[HTML]{FFFFFF} 55} & \multicolumn{1}{l|}{\cellcolor[HTML]{000000}{\color[HTML]{FFFFFF} 42}} & \multicolumn{1}{l|}{\cellcolor[HTML]{000000}{\color[HTML]{FFFFFF} 29}} & \multicolumn{1}{l|}{\cellcolor[HTML]{000000}{\color[HTML]{FFFFFF} 16}} & \multicolumn{1}{l|}{\cellcolor[HTML]{000000}{\color[HTML]{FFFFFF} 3}}  & 10                                                                     & 23                                                                     & 36 \\ \cline{1-5}
        {\color[HTML]{FFFFFF} 38} & \multicolumn{1}{l|}{\cellcolor[HTML]{000000}{\color[HTML]{FFFFFF} 25}} & \multicolumn{1}{l|}{\cellcolor[HTML]{000000}{\color[HTML]{FFFFFF} 12}} & 1                                                                      & 14                                                                     & 27                                                                     & 40                                                                     & 53 \\ \cline{1-3}
        {\color[HTML]{FFFFFF} 21} & \multicolumn{1}{l|}{\cellcolor[HTML]{000000}{\color[HTML]{FFFFFF} 8}}  & 5                                                                      & 18                                                                     & 31                                                                     & 44                                                                     & 57                                                                     & 70 \\ \cline{1-2}
        {\color[HTML]{FFFFFF} 4}  & 9                                                                      & 22                                                                     & 35                                                                     & 48                                                                     & 61                                                                     & 74                                                                     & 87 \\ \cline{1-1}
    \end{tabular}
    \caption{(13, 17) Yin-Yang Diagram}
\end{figure}
For a path $P$ from $(0,0)$ to $(n, m)$, let $L_P$ be the region of the
Yin-Yang diagram bounded above by $P$. Then there is a map from paths to bar partitions, given by
\begin{equation}
    \label{eq:pathtocore}
    P\mapsto L_P\ \Delta\ L_{P_0},
\end{equation}
where $P_0$ denotes the path from $(0, 0)$ to $(n, m)$ separating the Yin and Yang
regions and $\Delta(S, T) = (S\union T) \backslash (S\intersect T)$ denotes the symmetric difference of $S$ and $T$.
(Strictly speaking, this gives a map into sets, but a bar partition $\lambda$ can be represented uniquely as a set because all of its parts are distinct.)
\begin{lemma}
    The map given in $\eqref{eq:pathtocore}$ maps into $(\bar{s}, \bar{t})$-core paritions, and gives a bijection between $(\bar{s}, \bar{t})$-cores and monotonic paths in the Yin-Yang diagram.
\end{lemma}
\begin{proof}
    See~\cite{Bessenrodt20063}.
\end{proof}
\begin{definition}
For $S$ a subset of the positive integers, let $E(S) = S\intersect2\BZ$ (i.e.
the subset consisting of all of the even integers of $S$.) Define the
\textit{parity} of a path $P$ to be the parity of $E(L_P)$ (i.e. a path is even if $E(L_P)$ is even and odd if $E(L_P)$ is odd.)
\end{definition}
\begin{lemma}
    Even $(\bar{s}, \bar{t})$-core partitions correspond exactly to paths $P$ that have the same parity as $P_0$.
\end{lemma}
\begin{proof}
    We have
    \begin{align*}
        \abs{E(L_P \Delta L_{P_0})} &= \abs{E(L_P) \Delta E(L_{P_0})} \\ 
        &= \abs{E(L_P)} + \abs{E(L_{P_0})} - 2\abs{E(L_P)\intersect E(L_{P_0})}
    \end{align*}
    so $\abs{E(L_P \Delta L_{P_0})}$ is even if and only if $\abs{E(L_P)}\equiv \abs{E(L_{P_0})}\pmod{2}$, as desired.
\end{proof}
\begin{lemma}
\label{jacobilemma}
    Suppose $s \equiv 1\pmod{4}$. Then ${(-1)}^{\abs{E(L_{P_0})}} = \left( \frac{s}{t} \right)$, where $\jacobi{s}{t}$ denotes the Jacobi symbol.
\end{lemma}
\begin{proof}
    Let $c_j$ be the number of even integers in the $j$th column of the Yang half of the diagram. Once again, we will associate each square in the Yin-Yang diagram with the coordinates of its upper-right corner. Then for each $j$, the squares in the Yang half of the diagram with $x = j$ corrspond to those satisfying $1\le y \le \fl{\frac{sx}{t}}$. If $j$ is odd, then the even squares correspond to squres whose $y$ coordinate is odd, so
    \[
        c_j = \fl{\frac{\fl{\frac{sj}{t}} + 1}{2}} = \fl{\frac{\frac{sj}{t} + 1}{2}}.
    \]
    On the other hand, if $j$ is even, then the even squares in the $j$th column correpond to squares whose $y$ coordinate is even, so in this case,
    \[
        c_j = \fl{\frac{\fl{\frac{sj}{t}}}{2}} = \fl{\frac{\frac{sj}{t}}{2}}.
    \]
    In either case, the number of even numbers in the $j$th column is
    \[
        c_j = \fl{\frac{\frac{sj}{t} + j}{2}} - \fl{\frac{j}{2}} = \fl{\frac{\left( \frac{s}{t} + 1 \right)j}{2}} - \fl{\frac{j}{2}}.
    \]
    Thus
    \begin{equation}
        \label{eq:esum}
        \abs{E(L_{P_0})} = \dsum_{j=1}^{\frac{t-1}{2}} c_j = \dsum_{j = 1}^{\frac{t - 1}{2}}\fl{\frac{\left( \frac{s}{t} + 1 \right)j}{2}} - \fl{\frac{j}{2}}.
    \end{equation}
    Since $t$ is odd, we have
    \begin{equation}
        \fl{\frac{j}{2}} + \fl{\frac{t - j}{2}} = \frac{t - 1}{2},
    \end{equation}
    and for all $j$ not divisibly by $t$, we have.
    \begin{equation}
        \fl{\frac{\left( \frac{s}{t} + 1 \right)j}{2}} + \fl{\frac{\left( \frac{s}{t} + 1 \right)(t - j)}{2}} = \frac{s + t}{2} - 1.
    \end{equation}
    Thus
    \begin{equation}
        \label{eq:fldiffev}
        \fl{\frac{\left( \frac{s}{t} + 1 \right)j}{2}} - \fl{\frac{j}{2}} + \fl{\frac{\left( \frac{s}{t} + 1 \right)(t - j)}{2}} - \fl{\frac{t - j}{2}} = \frac{s - 1}{2}.
    \end{equation}
    Since $s\equiv 1\pmod{4}$, this means that the right hand side of \eqref{eq:fldiffev} is even, which means that in \eqref{eq:esum}, we can replace the $j = 1,2,\ldots,\frac{t-1}{2}$ with $j = 2, 4, \ldots, t - 1$. Indeed, for every $1\le j\le \frac{t - 1}{2}$, we can replace $j$ with the unique even residue in $\{j, t - j\}$ without changing the parity of the summation. Substituting this back into $\eqref{eq:esum}$ gives
    \begin{equation}
        E(s, t) \equiv \dsum_{k = 1}^{\frac{t - 1}{2}}\fl{\frac{\left( \frac{s}{t} + 1 \right)(2k)}{2}} - \fl{\frac{(2k)}{2}} \equiv \dsum_{k = 1}^{\frac{t - 1}{2}}\fl{\frac{sk}{t}}  \pmod{2}.
    \end{equation}
    However,~\cite{2000} showed that for odd integers $s, t$, 
    \begin{equation}
        \jacobi{s}{t} = (-1)^{\displaystyle\sum_{i=1}^{(t-1)/2}\fl{\frac{is}{t}}}.
    \end{equation}
    Thus $(-1)^{E(s,t)} = \jacobi{s}{t}$, as desired.
\end{proof}

\section{Counting Even and Odd Paths}
In this section, suppose $x, y$ are arbitrary positive integers. For each unit
square in the plane, color the square with upper right corner $(i,j)$ red if
$i+j$ is even; otherwise, color the square blue. In this case, define a path from $(0,0)$ to $(x,y)$ to be \textit{even} if the number of red squares between the path and the $x$-axis is even and odd otherwise. Note that in the case of the Yin-Yang diagram, red squares correspond to squares that contain an even integer, so the definitions agree.
\begin{lemma}
    \label{pathlemma}
    Let $D(x, y)$ be the number of even paths from $(0,0)$ to $(x, y)$ minus the number of odd paths from $(0,0)$ to $(x,y)$.
    \begin{itemize}
        \item If $x = 2k, y = 2l$, then $D(x, y) = \binom{k + l}{k}$.
        \item If $x = 2k, y = 2l + 1$, then $D(x, y) = \binom{k + l}{k}$.
        \item If $x = 2k + 1, y = 2l$, then $D(x, y) = (-1)^l\binom{k + l}{k}$.
        \item If $x = 2k + 1, y = 2l + 1$, then $D(x, y) = 0$.
    \end{itemize}
\end{lemma}
\begin{proof}
    We induct on $\min(x, y)$. Clearly $D(x, y) = 1$ if $x = 0$ or $y = 0$, and this agrees with the formulas above.
    
    Otherwise, we have $D(x, y) = D(x, y - 1) + (-1)^{c_x(y)}D(x - 1, y)$,
    where $c_x(y)$ is the number of red squares in the rectangle $[x-1,
    x]\times[0,y]$. As before, we have that $c_x(y) = \fl{\frac{x + y}{2}} -
    \fl{\frac{x}{y}}$. Note that $c_x(y) = \fl{\frac{y}{2}}$ unless $x$ and $y$
    are both odd, in which case it is equal to $\fl{\frac{y}{2}} + 1$. There
    are four cases for the parities of $x$ and $y$:
    \begin{itemize}
        \item $x = 2k, y = 2l$.
            \begin{align*}
                D(x,y) &= D(2k, 2l - 1) + (-1)^{l} D(2k - 1, 2l)\\
                &= \binom{k + l - 1}{l - 1} + (-1)^{l}\cdot (-1)^l\binom{(k - 1) + l}{l}\\
                &= \binom{k + l}{l}
            \end{align*}
        \item $x = 2k, y = 2l + 1$.
            \begin{align*}
                D(x,y) &= D(2k, 2l) + (-1)^{l} D(2k - 1, 2l + 1)\\
                &= \binom{k + l}{l} + 0\\
                &= \binom{k + 1}{l}
            \end{align*}
        \item $x = 2k + 1, y = 2l$.
            \begin{align*}
                D(x,y) &= D(2k + 1, 2l - 1) + (-1)^{l} D(2k, 2l)\\
                &= 0 + (-1)^{l}\binom{k + l}{l}\\
                &= (-1)^{l}\binom{k + 1}{l}
            \end{align*}
        \item $x = 2k + 1, y = 2l + 1$.
            \begin{align*}
                D(x,y) &= D(2k + 1, 2l) + (-1)^{k + l + 1 - k} D(2k, 2l + 1)\\
                &= (-1)^l \binom{k + l}{l} + (-1)^{l + 1}\binom{k + l}{l}\\
                &= 0
            \end{align*}
    \end{itemize}
\end{proof}
In particular, we have that if $y$ is even, then $D(x, y) = (-1)^{xy/2}\binom{k + l}{l}$.
\section{Even $(\bar{s},\bar{t})$-cores}
We are now ready to enumerate the even $(\bar{s}, \bar{t})$-core partitions.
\begin{theorem}
    \label{mainthm}
    Let $s, t > 1$ be relatively prime odd integers, and set $m = \fl{s / 2}, n = \fl{t / 2}, a = \fl{s / 4}, b = \fl{t / 4}$. In addition, let $\displaystyle\jacobi{*}{*}$ denote the Jacobi symbol. Then the number of even $(\bar{s}, \bar{t})$-core partitions is
    \begin{itemize}
        \item $\frac{1}{2}\binom{m + n}{n}$ if $s, t\equiv 3(4)$,
        \item $\frac{1}{2}\left( \binom{m + n}{n} + (-1)^{mn/2} \left(\frac{s}{t}\right)\binom{a + b}{b}\right)$ otherwise.
    \end{itemize}
    When $s$ and $t$ are primes, this gives the number of self-associate spin characters in $\tilde{S}_n$ that are defect 0 for both $s$ and $t$.
\end{theorem}
\begin{proof}
    The number of even paths in the $(s, t)$ Yin-Yang diagram minus the number of odd paths is equal to $D(n, m)$ (as given in Lemma $\ref{pathlemma}$). Thus the number of even $(\bar{s}, \bar{t})$-core partitions minus the number of odd $(\bar{s}, \bar{t})$-core partitions is equal to $(-1)^{\abs{E(L_{P_0})}} D(n, m)$.
Since the overall number of $(\bar{s}, \bar{t})$-core partitions is $\binom{m+n}{n}$, this means that the number of even $(\bar{s}, \bar{t})$-core partitions is
    \[
        \frac{1}{2}\left( \binom{m + n}{n} + (-1)^{\abs{E(L_{P_0})}} \cdot D(n, m) \right).
    \]
    If $s,t$ are both $3\pmod{4}$, then $m$ and $n$ are both odd. Thus $D(n, m) = 0$, so the number of even $(\bar{s}, \bar{t})$-core partitions is 
    \begin{equation}
        \label{eq:final-ans1}
        \frac{1}{2}\binom{m + n}{n}.
    \end{equation}
    Otherwise, suppose without loss of generality that $s\equiv 1\pmod4$.
    Then $m$ is even, so by Lemma~\ref{pathlemma},
    \[
        D(n, m) = (-1)^{mn/2} \binom{\fl{m/2} + \fl{n/2}}{\fl{n/2}} = (-1)^{mn/2} \binom{a + b}{b}.
    \]
    In addition, by Lemma~\ref{jacobilemma},
    \[
        (-1)^{\abs{E(L_{P_0})}} = \jacobi{s}{t}.
    \]
    Thus the number of even $(\bar{s},\bar{t})$-core partitions in this case is
    \begin{equation}
        \label{eq:final-ans2}
        \frac{1}{2}\left( \binom{m + n}{n} + (-1)^{mn/2} \left(\frac{s}{t}\right)\binom{a + b}{b}\right).
    \end{equation}
    Note that by quadratic reciprocity, this formula is symmetric in $s$ and
    $t$, so \eqref{eq:final-ans2} does in fact give the number of even
    $(\bar{s}, \bar{t})$-core partitions as long as $s$ and $t$ are not both $3\pmod{4}$.
\end{proof}
\section{Future Direction}
Bessenrodt and Olsson~\cite{Bessenrodt20063} showed that if $p < q$ are odd primes, then the Yin half of the $(p, q)$ Yin-Yang diagram is in some sense the ``largest'' $(\bar{p}, \bar{q})$-core partition, and thus the maximum $n$ for which there exists an associate class of spin characters in $\tilde{S}_n$ with defect 0 for $p$ and $q$ is just the sum of the numbers in the Yin diagram. More precisely, they showed that any $(\bar{p}, \bar{q})$-core can be contained in the partition represented by the Yin half of the diagram. We can ask the same question for even $(\bar{p}, \bar{q})$-core partitions. 
\begin{conjecture}
    For any pair of distinct odd primes $p, q$, there exists a $(\bar{p}, \bar{q})$-core partition $\lambda$ such that any even $(\bar{p}, \bar{q})$-core partition is contained in $\lambda$.
\end{conjecture}
Regardless of whether the conjecture is true, we can also ask the following question: 
\begin{question}
    What is the maximum $n$ for which a self-associate character in
    $\tilde{S}_n$ of defect 0 for $p$ and $q$ exists?
\end{question}
\section*{Acknowledgements}
    This research was conducted at the University of Minnesota Duluth REU and was supported by NSF grant 1358659 and NSA grant H98230-13-1-0273. The author thanks Rishi Nath for suggesting the problem, Joe Gallian for supervision to research, and Ben Gunby for helpful comments on the manuscript.
\bibliographystyle{plain}
\bibliography{writeup}
\end{document}